\documentclass[11pt]{article}
\usepackage{amssymb,amsmath,amsfonts}
\usepackage{graphics}
\usepackage{graphicx}

\DeclareMathOperator{\hol}{Hol}
\def\mrd{\mathrm{d}}

\def\eqref#1{$(\ref{#1})$}

\newenvironment{proof}{\noindent {\em {Proof}}.}{$\square$
\medskip}
\newtheorem{theorem}{Theorem}[section]
\newtheorem{corollary}[theorem]{Corollary}
\newtheorem{lemma}[theorem]{Lemma}
\newtheorem{remark}[theorem]{Remark}
\newtheorem{proposition}[theorem]{Proposition}

\newtheorem{example}[theorem]{Example}
\newtheorem{question}[theorem]{Question}
\newtheorem{problem}[theorem]{Problem}

\newcommand{\disk}{\mathbb{D}}

\newcommand {\intd} {\int _ \mathbb{D}}

\newcommand{\bt}[1]{\begin{theorem}\label{#1}}
\newcommand{\bc}[1]{\begin{corollary}\label{#1}}
\newcommand{\bl}[1]{\begin{lemma}\label{#1}}
\newcommand{\bp}[1]{\begin{proposition}\label{#1}}
\newcommand{\be}[1]{\begin{example}\rm\label{#1}}
\newcommand{\bq}[1]{\begin{question}\rm\label{#1}}
\newcommand{\bprob}[1]{\begin{problem}\rm\label{#1}}
\newcommand{\beq}[1]{\begin{eqnarray}\label{#1}}
\newcommand{\br}[1]{\begin{remark}\rm\label{#1}}

\newcommand{\el}{\end{lemma}}
\newcommand{\ep}{\end{proposition}}
\newcommand{\ee}{\end{example}}
\newcommand{\eq}{\end{question}}
\newcommand{\eprob}{\end{problem}}
\newcommand{\eeq}{\end{eqnarray}}
\newcommand{\ed}{\end{definition}}
\newcommand{\et}{\end{theorem}}
\newcommand{\ec}{\end{corollary}}
\newcommand{\er}{\end{remark}}

\title{Hilbert-Schmidt composition-differentiation operators on the unit ball}

\author{Ali Abkar
\vspace{0.5in}\\
Department of Pure Mathematics, Faculty of Science,\\
Imam Khomeini International University, Qazvin 34149, Iran\\
\texttt{Email:~abkar@sci.ikiu.ac.ir}\\}

\date{}

\begin{document}
\maketitle \textbf{Abstract.}
We use the notion of radial derivative to introduce 
composition-differentiation operators on the Hardy and Bergman spaces of the unit ball and the polydisk. We seek for necessary and sufficient conditions on the inducing functions to ensure that the composition-differentiation operator is Hilbert-Schmidt.\\


\textbf{}{}

\section{Introduction}
Let $\mathbb{D}$ denote the open unit disk in the complex plane.
For $\alpha>-1$, we define the weighted Bergman space $A^2_\alpha(\disk)$
as the space of analytic functions $f$ in $\mathbb{D}$ for which
 $$\Vert f\Vert^2_{A^2_\alpha(\disk)}=\intd |f(z)|^2 dA_\alpha (z)<+\infty,$$ where $dA_\alpha (z)=(\alpha +1)(1-|z|^2)^\alpha dA(z)$, here
$dA(z)=\pi^{-1}dx\,dy$ is the
normalized area measure in the complex plane.
 It is well-known that $A^2_\alpha(\disk)$ equipped with the inner product
 $$\left <f,g\right >= \intd
f(z)\overline{g(z)}dA_\alpha (z),$$
is a functional Hilbert.
It follows from boundedness of the point evaluation functional together with the Riesz' representation theorem that
each $f\in A^2_\alpha(\disk)$ can be written as
$$f(w)=\langle f,K_w^\alpha\rangle =\intd f(z)\overline {K_w^\alpha(z)}
dA_\alpha (z),\quad w\in\disk,$$
where
$$K_w^\alpha(z)=\frac{1}{(1-z\overline w)^{\alpha+2}}$$
is the
reproducing kernel for $A^2_\alpha(\disk)$.
More information on the theory of Bergman spaces can be found in \cite{duren} and \cite{hkz}.\par
We mean by the polydisk the subset
$\mathbb{D}^n=\mathbb{D}\times \cdots \times \mathbb{D}$ of $\mathbb{C}^n$,
and by the unit ball the
set
$$\mathbb{B}_n=\{z=(z_1,...,z_n)\in\mathbb{C}^n: |z|^2=|z_1|^2+\cdots+|z_n|^2<1\}.$$
Let $\hol(\mathbb{B}_n)$ denote the space of all holomorphic functions on the unit ball.
The weighted Bergman space on the unit ball $\mathbb{B}_n$ is defined as
$$A^2_\alpha(\mathbb{B}_n)=\hol(\mathbb{B}_n)\cap L^2(\mathbb{B}_n, dv_\alpha)$$
where $dv_\alpha(z) =c_\alpha (1-|z|^2)^\alpha dv(z)$, and
$dv(z)$ is the normalized volume measure on $\mathbb{B}_n$ and $c_\alpha$ is a positive constant so that $v_\alpha (\mathbb{B}_n)=1$.
Let
$J =(j_1,\cdots, j_n)$ be a multi-index; this means that $J$ is an $n$-tuple of
nonnegative integers. In this case we write $J \ge 0$. We define $|J|=j_1+\cdots + j_n$.
It can be verified that the vectors
\begin{equation}\label{on-bergman-ball}
e_J(z)=\sqrt{\frac{\Gamma(n+|J|+\alpha+1)}{J!\, \Gamma(n+\alpha+1)}}\, z^J,\quad J=(j_1,...,j_n).
\end{equation}
form an orthonormal basis for $A_\alpha^2(\mathbb{B}_n)$. Moreover, the
reproducing kernel associated to the points $(z_1,...,z_n)$ and $w=(w_1,...,w_n)$ of the unit ball is given by (see \cite{rud} or \cite{zhu2}):
$$K^\alpha(z,w)=\frac{1}{(1-\langle z,w\rangle)^{n+1+\alpha}},$$
where $\langle z,w\rangle =\sum_{k=1}^n z_k\overline w_k$ is the usual inner product in $\mathbb{C}^n$.\par
Let $\mathbb{T}^n$ denote the distinguished boundary of $\mathbb{D}^n$, and let $\mathbb{S}_n$ denote the (topological) boundary of $\mathbb{B}_n$. Let $d\sigma$ denote the normalized Lebesgue measure on $\mathbb{T}^n$ or $\mathbb{S}_n$. The Hardy space $H^2(\disk^n)$ consists of holomorphic functions on the polydisk such that
$$\sup_{0<r<1}\int_{\mathbb{T}^n}|f(r\zeta)|^2 \mathrm{d}\sigma(\zeta)<\infty.$$
Similarly, the Hardy space $H^2(\mathbb{B}_n)$ consists of holomorphic functions in the unit ball such that
$$\sup_{0<r<1}\int_{\mathbb{S}_n}|f(r\zeta)|^2 \mathrm{d}\sigma(\zeta)<\infty.$$
It is known that for every function $f$ in the Hardy space (of $\disk^n$ or $\mathbb{B}_n$), the boundary function
$$f(\zeta)=\lim_{r\to 1^-}f(r\zeta)$$
is well-defined almost everywhere on ($\mathbb{T}^n$ or $\mathbb{S}_n$). Therefore, $H^2(\Omega)$ is a Hilbert space with the following inner product:
$$\langle f,g\rangle=\int_{\partial\Omega}f(\zeta)\overline{g(\zeta)}\mathrm{d}\sigma(\zeta),\qquad \Omega=\disk^n,\,\, \textrm{or}\,\,  \Omega=\mathbb{B}_n.$$
We should recall that for a multi-index $J$ we have (see \cite[Prop. 1.4.9]{rud}, or \cite[Lemma 1.11]{zhu2}):
$$\int_{\mathbb{S}_n}|z^J|^2\mrd \sigma(z)=\frac{(n-1)!\,J!}{(n-1+|J|)!},$$
from which it follows that the vectors
\begin{equation}\label{on-hardy-ball}
\sqrt{\frac{(n-1+|J|)!}{(n-1)!\,J!}}\, z^J,\quad J=(j_1,...,j_n),
\end{equation}
form an orthonormal basis for $H^2(\mathbb{B}_n)$.
Now, let $\varphi:\mathbb{D}^m\to \mathbb{D}^n$ be a holomorphic function;
$$\varphi(z_1,...,z_m)=\Big (\varphi_1(z_1,...,z_m),...,\varphi_n(z_1,...,z_m)\Big ).$$
In the operator theory of one complex variables, the composition operator is defined by $C_\varphi (f)=f\circ \varphi$ where $\varphi$ is an analytic self-map of the unit disk, and $f$ is a function in the Hardy or Bergman space. This notion can be easily generalized to several variables.
For a holomorphic function $\psi :\mathbb{D}^m\to \mathbb{C}$, the weighted composition operator $C_{\psi,\varphi}:A_\alpha^2(\disk^n)\to A_\beta^2(\disk^m)$ is defined by
$$C_{\psi,\varphi}(f)=\psi\cdot \left(f\circ \varphi\right).$$
This operator is a rather classical one, and was studied by several authors. It was proved by M. Stessin and Kehe Zhu that only in some instances the composition operator is bounded (see \cite{szhu1} and \cite{szhu2}).
\par
Now, the question arises as how to generalize the one-dimensional composition-differentiation operator $D_\varphi (f)=f^\prime\circ \varphi$ to several variable case. To do this, we
let $f\in \hol(\mathbb{B}_n)$ be a holomorphic function of several variables. For $t>0$, we consider the radial differential operator
$$R^tf(z)=\sum_{k=1}^\infty k^t f_k(z),$$
where $f(z)=\sum_{k=0}^\infty f_k(z)$ is the homogenous expansion of $f$.
We now introduce the weighted composition-differentiation operator
\begin{equation*}\label{one-variable}
{E}_{\psi,\varphi}(f)=\psi\cdot \left( (Rf) \circ \varphi \right),\quad f\in \hol(\mathbb{B}_n),
\end{equation*}
and the higher order weighted composition-differentiation operator
\begin{equation*}\label{higher-order}
{E}_{\psi,\varphi}^t(f)=\psi\cdot \left((R^tf) \circ \varphi\right),\quad t\ge 0,\,\,  f\in \hol(\mathbb{B}_n).
\end{equation*}
The main purpose of this paper is to study conditions on the symbol $\varphi$ and the function $\psi$ to ensure that the operator ${E}_{\psi,\varphi}^t$ is Hilbert-Schmidt. This will be done in section \S 3. We emphasize that ${E}_{\psi,\varphi}$ generalizes the one-variable composition-differentiation operator
$D_{\psi,\varphi}(f)=\psi\cdot (f^\prime \circ \varphi)$. This operator in one-dimensional case was studied by several people;
see for instance \cite{allen}, \cite{abkar1}, \cite{fat1}, \cite{fat2}, \cite{hib}, \cite{ohno2}, \cite{lia}, \cite{ohno1}.

\section{Composition operators}
This section is devoted to the study of Hilbert-Schmidt weighted composition operators. Our approach is based on direct computations using the standard
orthonormal basis of the space in question.
\begin{theorem}\label{bergman-dn-to-d}
Let $\varphi:\disk\to \disk^n$ be holomorphic, and let $\psi$ be an analytic function on the unit disk.
The operator $C_{\psi,\varphi}:A_\alpha^2(\disk^n)\to A_\beta^2(\disk)$, where $\beta=n(\alpha+2)-2$, is Hilbert-Schmidt if and only if
\begin{equation*}\label{eq1}
\int_{\mathbb{D}}\prod_{k=1}^n
\frac{|\psi(z)|^{2/n}}{(1-\left|\varphi_k(z)\right|^2)^{\alpha+2}}\mrd A_\beta(z)<\infty.
\end{equation*}
\end{theorem}
\begin{proof}
First assume that $\psi(z)\equiv 1$.
Consider the orthonormal basis
$\{e_J(z)\}_{J\ge 0}$, where
$$e_J(z)=\frac{z_1^{j_1}\cdots z_n^{j_n}}{\gamma_{j_1}\cdots\gamma_{j_n}}.$$
We then have
\begin{align*}
\|C_\varphi(e_J)\|_\beta^2 &=\|\frac{\varphi_1^{j_1}\cdots\varphi_n^{j_n}}{\gamma_{j_1}\cdots\gamma_{j_n}}\|^2_\beta\\
&=\int_{\disk}\frac{|\varphi_1|^{2j_1}}{\gamma^2_{j_1}}\cdots\frac{|\varphi_n|^{2j_n}}
{\gamma^2_{j_n}}\mrd A_\beta(z).
\end{align*}
This implies that
\begin{align*}
\sum_{J\ge 0}\|C_\varphi(e_J)\|^2_\beta &=\sum_{J\ge 0}
\int_{\disk}\frac{|\varphi_1(z)|^{2j_1}}{\gamma^2_{j_1}}\cdots \frac{|\varphi_n(z)|^{2j_n}}{\gamma^2_{j_n}}\mrd A_\beta(z)\\
&=\int_{\disk}\sum_{j_1=0}^\infty \frac{|\varphi_1(z)|^{2j_1}}{\gamma^2_{j_1}}\cdots\sum_{j_n=0}^\infty \frac{|\varphi_n(z)|^{2j_n}}{\gamma^2_{j_n}}\mrd A_\beta(z)\\
&=\int_{\disk}\frac{1}{(1-|\varphi_1(z)|^2)^{\alpha+2}}\cdots\frac{1}{(1-|\varphi_n(z)|^2)^{\alpha+2}}\mrd A_\beta(z).
\end{align*}
In general, assume that $\psi$ is a holomorphic function.
The above computations reveal that
$$\sum_{J\ge 0}\|C_{\psi, \varphi}(e_J)\|^2_\beta$$
is finite if and only if
\begin{equation*}\label{eq1}
\int_{\mathbb{D}}\prod_{k=1}^n
\frac{|\psi(z)|^{2/n}}{(1-\left|\varphi_k(z)\right|^2)^{\alpha+2}}\mrd A_\beta(z)<\infty.
\end{equation*}
\end{proof}

\noindent{\bf Remark.} The same argument as in the proof of Theorem \ref{bergman-dn-to-d} shows that if
$\varphi:\disk^m \to \disk^n$ is holomorphic, and if $\psi:\disk^m\to \disk$ is holomorphic, then
$C_{\psi,\varphi}:A_\alpha^2(\disk^n)\to A_\beta^2(\disk^m)$, where $\beta=n(\alpha+2)-2$, is Hilbert-Schmidt if and only if
\begin{equation*}
\int_{\mathbb{D}^m}\prod_{k=1}^n
\frac{|\psi(z)|^{2/n}}{(1-\left|\varphi_k(z)\right|^2)^{\alpha+2}}\mrd V_\beta(z)<\infty.
\end{equation*}

\begin{corollary}\cite[Theorem 33]{szhu2}
The operator $C_{\varphi}:A_\alpha^2(\disk^n)\to A_\beta^2(\disk)$ is Hilbert-Schmidt if and only if
\begin{equation*}\label{eq3}
\int_{\mathbb{D}}\prod_{k=1}^n
 \Big (\frac{1-|z|^2}{1-\left|\varphi_k(z)\right|^2}\Big )^{\alpha+2} \mrd \lambda(z)<\infty,
\end{equation*}
where
$$\mrd\lambda(z)=\frac{\mrd A(z)}{(1-|z|^2)^2}$$
is the  M\"{o}bius invariant measure on the unit disk.
\end{corollary}
\begin{proof}
For $\beta=n(\alpha+2)-2$, we have
\begin{align*}\mrd A_\beta (z)&=(\beta+1)(1-|z|^2)^{n(\alpha+2)-2}\mrd A(z)\\
&=(\beta+1)(1-|z|^2)^{n(\alpha+2)}\mrd \lambda(z),
\end{align*}
from which the result follows.
\end{proof}

\begin{theorem}\label{hardy-dn-to-d}
Let $\varphi:\disk\to \disk^n$ be holomorphic, and let $\psi$ be an analytic function on the unit disk.
Then operator $C_{\psi,\varphi}:H^2(\disk^n)\to A_\beta^2(\disk)$, where $\beta=n-2$, is Hilbert-Schmidt if and only if
\begin{equation*}\label{eq3}
\int_{\mathbb{D}}\prod_{k=1}^n\Big(
\frac{|\psi(z)|^{2/n}}{1-\left|\varphi_k(z)\right|^2}\Big)\mrd A_\beta(z)<\infty.
\end{equation*}
\end{theorem}
\begin{proof}
First assume that $\psi(z)\equiv 1$.
Consider the following orthonormal basis
$$e_J(z)=z_1^{j_1}\cdots z_n^{j_n},\quad J=(j_1,...,j_n),$$
for the Hardy space $H^2(\disk^n)$. We then have
\begin{align*}
\|C_\varphi(e_J)\|_\beta^2 &=
\|\varphi_1^{j_1}\cdots\varphi_n^{j_n}\|^2_\beta\\
&=\int_{\disk}|\varphi_1|^{2j_1}\cdots |\varphi_n|^{2j_n}\mrd A_\beta(z),
\end{align*}
which implies that
\begin{align*}
\sum_{J\ge 0}\|C_\varphi(e_J)\|^2_\beta &=\int_{\disk}\sum_{j_1=0}^\infty |\varphi_1(z)|^{2j_1} \cdots\sum_{j_n=0}^\infty |\varphi_n(z)|^{2j_n}\mrd A_\beta(z)\\
&=\int_{\disk}\prod_{k=1}^n \Big(\frac{1}{1-|\varphi_k(z)|^2}\Big )\mrd A_\beta(z).
\end{align*}
In general, assume that $\psi$ is a holomorphic function on the unit disk.
The above computations reveal that
$$\sum_{J\ge 0}\|C_{\psi, \varphi}(e_J)\|^2_\beta$$
is finite if and only if
\begin{equation*}\label{eq1}
\int_{\mathbb{D}}\prod_{k=1}^n
\frac{|\psi(z)|^{2/n}}{1-\left|\varphi_k(z)\right|^2}\mrd A_\beta(z)<\infty.
\end{equation*}
\end{proof}

\noindent{\bf Remark.} The same argument as in the proof of Theorem \ref{hardy-dn-to-d} shows that if
$\varphi:\disk^m \to \disk^n$ and $\psi:\disk^m\to \disk$ are holomorphic, then
$C_{\psi,\varphi}:H^2(\disk^n)\to A_\beta^2(\disk^m)$, where $\beta=n-2$, is Hilbert-Schmidt if and only if
\begin{equation*}\label{eq3}
\int_{\mathbb{D}^m}\prod_{k=1}^n\Big(
\frac{|\psi(z)|^{2/n}}{1-\left|\varphi_k(z)\right|^2}\Big)\mrd V_\beta(z)<\infty.
\end{equation*}

\begin{corollary}\cite[Theorem 32]{szhu2}
The operator $C_{\varphi}:H^2(\disk^n)\to A_\beta^2(\disk)$ is Hilbert-Schmidt if and only if
\begin{equation*}\label{eq4}
\int_{\mathbb{D}}\Big(\prod_{k=1}^n
 \frac{1-|z|^2}{1-\left|\varphi_k(z)\right|^2}\Big )\mrd \lambda(z)<\infty.
\end{equation*}
\end{corollary}
\begin{proof}
Since $\beta=n-2$, it follows that
\begin{align*}\mrd A_\beta (z)&=(n-1)(1-|z|^2)^{n-2}\mrd A(z)\\
&=(n-1)(1-|z|^2)^{n}\mrd \lambda(z),
\end{align*}
which in turn leads to the desired conclusion.
\end{proof}

As for the Hardy space of the unit ball $\mathbb{B}_n$, we should recall that for a multi-index $J$ we have (see \cite[Prop. 1.4.9]{rud}, or \cite[Lemma 1.11]{zhu2}):
$$\int_{\mathbb{S}_n}|z^J|^2\mrd \sigma(z)=\frac{(n-1)!\,J!}{(n-1+|J|)!}.$$
This implies that the vectors
\begin{equation}\label{on-hardy-ball}
\sqrt{\frac{(n-1+|J|)!}{(n-1)!\,J!}}\, z^J,\quad J=(j_1,...,j_n),
\end{equation}
form an orthonormal basis for $H^2(\mathbb{B}_n)$.

\begin{theorem}\label{hardy-bn-to-d}
Let $n>1$ and let $\varphi:\disk\to \mathbb{B}_n$ be a holomorphic function.
Then the operator $C_{\psi,\varphi}:H^2(\mathbb{B}_n)\to A_\beta^2(\disk)$, where $\beta=n-2$, is Hilbert-Schmidt if and only if
\begin{equation*}\label{eq3}
\int_{\mathbb{D}}
\frac{|\psi(z)|^{2}}{\big(1-\left|\varphi(z)\right|^2\big)^n}\,\mrd A_\beta(z)<\infty.
\end{equation*}
\end{theorem}
\begin{proof}
We consider the orthonormal functions defined by \eqref{on-hardy-ball}.
For $\psi$ identically $1$ we have
\begin{align*}
\|C_\varphi(e_J)\|_\beta^2 &=
\frac{(n-1+|J|)!}{(n-1)!\,J!}\|\varphi_1^{j_1}\cdots\varphi_n^{j_n}\|^2_\beta\\
&=\frac{(n-1+|J|)!}{(n-1)!\,J!}\int_{\disk}|\varphi_1|^{2j_1}\cdots |\varphi_n|^{2j_n}\mrd A_\beta(z).
\end{align*}
Therefore
\begin{align*}
\sum_{J\ge 0}\|C_\varphi(e_J)\|^2_\beta &=\int_{\disk}\sum_{J\ge 0}\frac{(n-1+|J|)!}{(n-1)!\,J!}\left(|\varphi_1|^{2}\cdots |\varphi_n|^{2}\right)^J\mrd A_\beta(z)\\
&=\int_{\disk}\sum_{k=0}^\infty \sum_{|J|=k}\frac{(k+n-1)!}{(n-1)!}\frac{\left(|\varphi_1|^{2}\cdots |\varphi_n|^{2}\right)^J}{J!}\mrd A_\beta(z).
\end{align*}
On the other hand, it follows from the multi-nomial formula that
\begin{equation}\label{multi}
\sum_{|J|=k}\frac{\left(|\varphi_1|^{2}\cdots |\varphi_n|^{2}\right)^J}{J!}=\frac{\left(|\varphi_1|^{2}+\cdots+ |\varphi_n|^{2}\right)^k}{k!}=\frac{|\varphi|^{2k}}{k!}.
\end{equation}
This, in turn, leads to
\begin{align*}
\sum_{J\ge 0}\|C_\varphi(e_J)\|^2_\beta &=\int_{\disk}\sum_{k=0}^\infty\frac{(k+n-1)!}{(n-1)!\, k!}|\varphi(z)|^{2k}\mrd A_\beta(z)\\
&=\int_{\disk}\sum_{k=0}^\infty\frac{(k+1)(k+2)\cdots (k+n-1)}{(n-1)!}|\varphi(z)|^{2k}\mrd A_\beta(z)\\
&=\int_{\disk}\frac{1}{(1-|\varphi(z)|^2)^n}\mrd A_\beta(z).
\end{align*}
In general, we obtain
$$\sum_{J\ge 0}\|C_{\psi,\varphi}(e_J)\|^2_\beta=\int_{\mathbb{D}}
\frac{|\psi(z)|^{2}}{\big(1-\left|\varphi(z)\right|^2\big)^n}\,\mrd A_\beta(z),$$
from which the result follows.
\end{proof}

\begin{corollary}\cite[Theorem 30]{szhu2}
$C_{\varphi}:H^2(\mathbb{B}_n)\to A_\beta^2(\disk)$ is Hilbert-Schmidt if and only if
\begin{equation*}\label{eq5}
\int_{\mathbb{D}}
 \left(\frac{1-|z|^2}{1-\left|\varphi(z)\right|^2}\right)^n \mrd \lambda(z)<\infty.
\end{equation*}
\end{corollary}
\begin{proof}
As in the preceding case, we have
\begin{align*}
\mrd A_\beta (z)=(n-1)(1-|z|^2)^{n}\mrd \lambda(z),
\end{align*}
from which the result follows.
\end{proof}

\begin{theorem}\label{thm3}
Let $\alpha>-1$ and let $\varphi:\disk\to \mathbb{B}_n$ be a holomorphic function.
Then the operator $C_{\psi,\varphi}:A_\alpha^2(\mathbb{B}_n)\to A_\beta^2(\disk)$, where $\beta=n-1+\alpha$, is Hilbert-Schmidt if and only if
\begin{equation*}\label{eq11}
\int_{\mathbb{D}}
\frac{|\psi(z)|^{2}}{\big(1-\left|\varphi(z)\right|^2\big)^{n+\alpha+1}}\,\mrd A_\beta(z)<\infty.
\end{equation*}
\end{theorem}
\begin{proof}
It is easy to see that the vectors
\begin{equation}\label{on-bergman-ball}
e_J(z)=\sqrt{\frac{\Gamma(n+|J|+\alpha+1)}{J!\, \Gamma(n+\alpha+1)}}\, z^J,\quad J=(j_1,...,j_n).
\end{equation}
form an orthonormal basis for $A_\alpha^2(\mathbb{B}_n)$.
As in the preceding theorem we have
\begin{align*}
\|C_\varphi(e_J)\|_\beta^2=
\frac{\Gamma(n+|J|+\alpha+1)}{\Gamma(n+\alpha+1)\,J!}\int_{\disk}|\varphi_1|^{2j_1}\cdots |\varphi_n|^{2j_n}\mrd A_\beta(z).
\end{align*}
Therefore
\begin{align*}
\sum_{J\ge 0}\|C_\varphi(e_J)\|^2_\beta &=
\int_{\disk}\sum_{k=0}^\infty \sum_{|J|=k}\frac{\Gamma(k+n+\alpha+1)}{\Gamma(n+\alpha+1)}\frac{\left(|\varphi_1|^{2}\cdots |\varphi_n|^{2}\right)^J}{J!}\mrd A_\beta(z)\\
&=\int_{\disk}\sum_{k=0}^\infty \frac{\Gamma(k+n+\alpha+1)}{\Gamma(n+\alpha+1)\, k!}|\varphi|^{2k}\mrd A_\beta(z)\\
&=\int_{\disk}\frac{1}{\big(1-|\varphi(z)|^2\big)^{n+\alpha+1}}\mrd A_\beta(z).
\end{align*}
This implies that
$\sum_{J\ge 0}\|C_{\psi,\varphi}(e_J)\|^2_\beta$ is finite if and only if
$$\int_{\disk}\frac{1}{\big(1-|\varphi(z)|^2\big)^{n+\alpha+1}}\mrd A_\beta(z)<\infty.$$
\end{proof}
\begin{corollary}\cite[Theorem 31]{szhu2}.
Let $\alpha>-1$ and let $\varphi:\disk\to \mathbb{B}_n$ be a holomorphic function.
Then the operator $C_{\varphi}:A_\alpha^2(\mathbb{B}_n)\to A_{n-1+\alpha}^2(\disk)$ is Hilbert-Schmidt if and only if
\begin{equation*}\label{eq12}
\int_{\mathbb{D}}
\bigg(\frac{1-|z|^2}{1-\left|\varphi(z)\right|^2}\bigg)^{n+\alpha+1}\,\mrd \lambda(z)<\infty.
\end{equation*}
\end{corollary}
\begin{proof}
Note that in this case
$$\mrd A_\beta(z)=(n+\alpha)(1-|z|^2)^{n+\alpha+1}\mrd \lambda(z).$$
\end{proof}

\section{Composition-differentiation operators}
When dealing with the theory of composition operators in the unit disk, one usually considers the composition-differentiation operator defined by
$$D_\varphi (f)=f^\prime\circ \varphi.$$
Some authors consider $D_\varphi$ as the composition of two successive operators $C_\varphi$ and $D$ where $D(f)=f^\prime$ is the differentiation operator. For this reason, they use the notation $C_\varphi D$ for what we denoted by $D_\varphi$. This operator was studied in detail by several authors; see for instance \cite{abkar1}, \cite{fat1}, \cite{fat2}, \cite{ohno2}, \cite{lia}, \cite{ohno1}.
\par
To generalize this definition to several complex variables, one encounters with the following two notions
of differentiation. The first one is to use the gradient of $f$, that is,
$$\nabla f=\left(\frac{\partial f}{\partial z_1},\cdots, \frac{\partial f}{\partial z_n}\right).$$
Using this concept, one may define
$$G_\varphi (f)=(\nabla f)\circ \varphi=\left(\frac{\partial f}{\partial z_1}\circ\varphi,\cdots, \frac{\partial f}{\partial z_n}\circ\varphi\right).$$
It is also possible to use the invariant gradient $\widetilde{\nabla}f$ instead of the gradient of $f$.
Another approach is to use the notion of radial derivative of a holomorphic function. It turns out that the second approach is more suitable for our purposes.
For a function $f$ holomorphic in the unit ball $\mathbb{B}_n$, we write
$$Rf(z)=\sum_{k=1}^n z_k\frac{\partial f}{\partial z_k}(z).$$
It is easy to verify that if
$$f(z)=\sum_{k=0}^\infty f_k(z)$$
is the homogenous expansion of $f$, then
$$Rf(z)=\sum_{k=1}^\infty kf_k(z).$$
Note that if
$f(z)=e_J(z)=z_1^{j_1}\cdots z_n^{j_n}$ is a multi-nomial,
then
$f(z)=f_k(z)$ where $k=|J|$. This implies that
\begin{equation}\label{rfphi}
Rf(\varphi(z))=k\,e_J(\varphi(z))=|J|\, \varphi_1^{j_1}\cdots\varphi_n^{j_n}.
\end{equation}
This identity reveals that the radial derivative is a good candidate to play the role of $f^\prime$ when $f$ is a
holomorphic function of several complex variables.
\par
As in the case of usual differentiation, one can define higher order radial derivatives. Let $t$ be a real number.
We define the radial differential operator $R^t$ as follows:
$$R^tf(z)=\sum_{k=1}^\infty k^t f_k(z),\quad f(z)=\sum_{k=0}^\infty f_k(z).$$
We now define the $n$-variable weighted composition-differentiation operator
\begin{equation}\label{one-variable}
{E}_{\psi,\varphi}(f)=\psi\cdot \left((Rf) \circ \varphi\right),\quad f\in \hol(\mathbb{B}_n),
\end{equation}
and the higher order $n$-variable weighted composition-differentiation operator
\begin{equation}\label{higher-order}
{E}_{\psi,\varphi}^t(f)=\psi\cdot \left((R^tf) \circ \varphi\right),\quad t\ge 0,\,\,  f\in \hol(\mathbb{B}_n).
\end{equation}
It is easy to see that if
$f(z)=e_J(z)=z_1^{j_1}\cdots z_n^{j_n}$,
then
\begin{equation}\label{rtfphi}
R^tf(\varphi(z))=k^t\, \varphi_1^{j_1}\cdots\varphi_n^{j_n},\quad k=|J|.
\end{equation}
For more information on the concept of fractional radial derivatives and their applications in function theory of several complex variables, we refer the reader to the book authored by Kehe Zhu \cite{zhu2}; who first introduced this notion in \cite{kzhu}.\par
In the following, we aim to study the conditions that are equivalent to Hilbert-Schmidtness of the operators ${E}_{\psi,\varphi}^t$ for $t\ge 0$.
Before we shift to our main discussion, it is instructive to see how the weighted composition-differentiation operator
$$D_{\psi,\varphi}(f):=\psi(z) f^\prime (\varphi(z))$$
behaves in one-dimensional case. We recall that an operator $T$ defined from a Hilbert space $H_1$ to another Hilbert space $H_2$ is called Hilbert-Schmidt if for some orthonormal basis $\{e_n\}$ in $H_1$ we have $\sum_n \Vert Te_n\Vert ^2<\infty$.

\begin{theorem}\label{one-dimension}
Let $\varphi$ be an analytic self-map of the unit disk, and $\psi$ be an analytic function $\disk$. Then $D_{\psi,\varphi}$ is a Hilbert-Schmidt operator on $A_\alpha^2(\disk)$ if and only if
$$\int_\mathbb{D}
\frac{|\psi(z)|^2}{\big(1-\left |\varphi(z)\right|^2\big)^{4+\alpha}}\mrd A_\alpha(z)<\infty.$$
\end{theorem}
\begin{proof}
We consider the orthonormal basis that consists of the unit vectors
$$e_k(z)=\sqrt{\frac{\Gamma(k+\alpha+2)}{k!\,\Gamma(\alpha+2)}}\,z^k,\quad k=0,1,2,\cdots.$$
We now have
\begin{align*}
\sum_{k=0}^\infty \Vert D_{\psi,\varphi}(e_k)\Vert^2_{A^2_\alpha}&=
\sum_{k=1}^\infty\int_\mathbb{D}\bigg |\psi(z)
 k\sqrt{\frac{\Gamma(k+\alpha+2)}{k!\,\Gamma(\alpha+2)}}\,\varphi(z)^{k-1}\bigg |^2\mrd A_\alpha(z)\\
&=\int_\mathbb{D}|\psi(z)|^2
\sum_{k=1}^\infty \frac{k\,\Gamma(k+\alpha+2)}{(k-1)!\,\Gamma(\alpha+2)}
\left|\varphi(z)\right |^{2(k-1)}\mrd A_\alpha(z)\\
&=\int_\mathbb{D}|\psi(z)|^2
\sum_{k=0}^\infty \frac{(k+1)\,\Gamma(k+\alpha+3)}{k!\,\Gamma(\alpha+2)}
\left|\varphi(z)\right |^{2k}\mrd A_\alpha(z)
\end{align*}
We now use the fact that (for $x>0$)
$$\Gamma(k+x)\asymp k^x\,\Gamma(x),\quad \, k\to \infty,$$
to conclude that
\begin{align*}
\sum_{k=0}^\infty \Vert D_{\psi,\varphi}(e_k)\Vert^2_{A^2_\alpha}&=
\int_\mathbb{D}|\psi(z)|^2
\sum_{k=0}^\infty \frac{(k+1)\,\Gamma(k+\alpha+3)}{k!\,\Gamma(\alpha+2)}
\left|\varphi(z)\right |^{2k}\mrd A_\alpha(z)\\
&\asymp
\int_\mathbb{D}|\psi(z)|^2
\sum_{k=0}^\infty \frac{\Gamma(k+\alpha+4)}{k!\,\Gamma(\alpha+4)}
\left|\varphi(z)\right |^{2k}\mrd A_\alpha(z)\\
&=\int_\mathbb{D}
\frac{|\psi(z)|^2}
{\big(1-\left|\varphi(z)\right|^{2}\big)^{\alpha+4}}
\mrd A_\alpha(z).
\end{align*}
This completes the proof.
\end{proof}

We now return to the $n$-variable definition of weighted composition-differentiation operator
$$E_{\psi,\varphi}(f)=\psi\cdot \left((Rf) \circ \varphi\right),\quad f\in \hol(\mathbb{B}_n).$$
\begin{theorem}\label{rf-bergman-ball}
Let $\alpha>-1$ and let $\varphi:\disk\to \mathbb{B}_n$ be a holomorphic function.
Then the operator $E_{\psi,\varphi}:A_\alpha^2(\mathbb{B}_n)\to A_\beta^2(\disk)$, where $\beta=n-1+\alpha$, is Hilbert-Schmidt if and only if
\begin{equation*}\label{eq21}
\int_{\mathbb{D}}
\frac{|\psi(z)|^{2}}{\big(1-\left|\varphi(z)\right|^2\big)^{n+\alpha+3}}\,\mrd A_\beta(z)<\infty.
\end{equation*}
\end{theorem}
\begin{proof}
We start by the orthonormal vectors \eqref{on-bergman-ball}
for $A_\alpha^2(\mathbb{B}_n)$. It follows that
\begin{align*}
\|E_{\psi,\varphi}(e_J)\|_\beta^2=
\frac{|J|^2\,\Gamma(n+|J|+\alpha+1)}{\Gamma(n+\alpha+1)\,J!}\int_{\disk}|\psi(z)|^2\left(|\varphi_1|^{2j_1}\cdots |\varphi_n|^{2j_n}\right)\mrd A_\beta(z).
\end{align*}
Using the multi-nomial formula
\begin{equation}\label{multi}
\sum_{|J|=k}\frac{\left(|\varphi_1|^{2}\cdots |\varphi_n|^{2}\right)^J}{J!}=\frac{\left(|\varphi_1|^{2}+\cdots+ |\varphi_n|^{2}\right)^k}{k!}=\frac{|\varphi|^{2k}}{k!},
\end{equation}
 and \eqref{rfphi} we obtain
\begin{align*}
\sum_{J\ge 0}\|E_{\psi,\varphi}(e_J)\|^2_\beta &=
\int_{\disk}|\psi(z)|^2\sum_{k=0}^\infty \sum_{|J|=k}\frac{|J|^2\,\Gamma(|J|+n+\alpha+1)}{\Gamma(n+\alpha+1)}\frac{\left(|\varphi_1|^{2}\cdots |\varphi_n|^{2}\right)^J}{J!}\mrd A_\beta(z)\\
&=\int_{\disk}|\psi(z)|^2 \sum_{k=0}^\infty \frac{k^2\,\Gamma(k+n+\alpha+1)}{\Gamma(n+\alpha+1)\, k!}|\varphi|^{2k}\mrd A_\beta(z)\\
&\asymp\int_{\disk}|\psi(z)|^2 \sum_{k=0}^\infty \frac{\Gamma(k+n+\alpha+3)}{\Gamma(n+\alpha+3)\, k!}|\varphi|^{2k}\mrd A_\beta(z)\\
&=\int_{\disk}\frac{|\psi(z)|^2}{\big(1-|\varphi(z)|^2\big)^{n+\alpha+3}}\mrd A_\beta(z),
\end{align*}
from which the result follows.
\end{proof}

Note that if $n=1$, then $\beta=\alpha$, so that the above result reduces to the one obtained in Theorem \ref{one-dimension}. This phenomenon confirms that the new notion of radial differentiation is a good generalization of the usual concept of differentiation.

\begin{theorem}\label{rf-hardy-ball}
Let $n>1$ and let $\varphi:\disk\to \mathbb{B}_n$ be a holomorphic function.
Then the operator $E_{\psi,\varphi}:H^2(\mathbb{B}_n)\to A_\beta^2(\disk)$, where $\beta=n-2$, is Hilbert-Schmidt if and only if
\begin{equation*}\label{eq23}
\int_{\mathbb{D}}
\frac{|\psi(z)|^{2}}{\big(1-\left|\varphi(z)\right|^2\big)^{n+2}}\,\mrd A_\beta(z)<\infty.
\end{equation*}
\end{theorem}
\begin{proof}
For the orthonormal functions \eqref{on-hardy-ball}
in $H^2(\mathbb{B}_n)$ we have
\begin{align*}
\|E_{\psi,\varphi}(e_J)\|_\beta^2 &=\frac{|J|^2\, (n-1+|J|)!}{(n-1)!\,J!}\int_{\disk}|\psi(z)|^2\left(|\varphi_1|^{2}\cdots |\varphi_n|^{2}\right)^J\mrd A_\beta(z).
\end{align*}
Therefore,
\begin{align*}
\sum_{J\ge 0}\|E_{\psi,\varphi}(e_J)\|^2_\beta &=\int_{\disk}|\psi(z)|^2\sum_{k=0}^\infty\frac{k^2\,(k+n-1)!}{(n-1)!\, k!}|\varphi(z)|^{2k}\mrd A_\beta(z)\\
&\asymp\int_{\disk}|\psi(z)|^2\sum_{k=0}^\infty\frac{(k+n+1)!}{(n+1)!\, k!}|\varphi(z)|^{2k}\mrd A_\beta(z)\\
&=\int_{\disk}\frac{|\psi(z)|^2}{(1-|\varphi(z)|^2)^{n+2}}\mrd A_\beta(z).
\end{align*}
\end{proof}

We now take up the composition-differentiation operators \eqref{higher-order} defined by higher order radial derivatives.
\begin{theorem}\label{rtf-bergman-ball}
Let $\alpha>-1$ and let $\varphi:\disk\to \mathbb{B}_n$ be a holomorphic function.
Then the operator $E^t_{\psi,\varphi}:A_\alpha^2(\mathbb{B}_n)\to A_\beta^2(\disk)$, where $\beta=n-1+\alpha$, is Hilbert-Schmidt if and only if
\begin{equation*}\label{eq21}
\int_{\mathbb{D}}
\frac{|\psi(z)|^{2}}{\big(1-\left|\varphi(z)\right|^2\big)^{n+\alpha+2t+1}}\,\mrd A_\beta(z)<\infty.
\end{equation*}
\end{theorem}
\begin{proof}
We start by the orthonormal vectors \eqref{on-bergman-ball}.
It follows from \eqref{rtfphi} that
\begin{align*}
\|E_{\psi,\varphi}(e_J)\|_\beta^2=
\frac{|J|^{2t}\,\Gamma(n+|J|+\alpha+1)}{\Gamma(n+\alpha+1)\,J!}\int_{\disk}|\psi(z)|^2\left(|\varphi_1|^{2j_1}\cdots |\varphi_n|^{2j_n}\right)\mrd A_\beta(z).
\end{align*}
Therefore from \eqref{multi} we obtain
\begin{align*}
\sum_{J\ge 0}\|E_{\psi,\varphi}(e_J)\|^2_\beta &=
\int_{\disk}|\psi(z)|^2 \sum_{k=0}^\infty \frac{k^{2t}\,\Gamma(k+n+\alpha+1)}{\Gamma(n+\alpha+1)\, k!}|\varphi|^{2k}\mrd A_\beta(z)\\
&\asymp\int_{\disk}|\psi(z)|^2 \sum_{k=0}^\infty \frac{\Gamma(k+n+\alpha+2t+1)}{\Gamma(n+\alpha+2t+1)\, k!}|\varphi|^{2k}\mrd A_\beta(z)\\
&=\int_{\disk}\frac{|\psi(z)|^2}{\big(1-|\varphi(z)|^2\big)^{n+\alpha+2t+1}}\mrd A_\beta(z),
\end{align*}
from which the result follows.
\end{proof}

\begin{theorem}\label{rtf-hardy-ball}
Let $n>1$ and let $\varphi:\disk\to \mathbb{B}_n$ be a holomorphic function.
Then the operator $E^t_{\psi,\varphi}:H^2(\mathbb{B}_n)\to A_\beta^2(\disk)$, where $\beta=n-2$, is Hilbert-Schmidt if and only if
\begin{equation*}\label{eq23}
\int_{\mathbb{D}}
\frac{|\psi(z)|^{2}}{\big(1-\left|\varphi(z)\right|^2\big)^{n+2t}}\,\mrd A_\beta(z)<\infty.
\end{equation*}
\end{theorem}
\begin{proof}
Choosing $e_J(z)$ as in \eqref{on-hardy-ball},
we have
\begin{align*}
\|E^t_{\psi,\varphi}(e_J)\|_\beta^2 &=\frac{|J|^{2t}\,(n-1+|J|)!}{(n-1)!\,J!}\int_{\disk}|\psi(z)|^2\left(|\varphi_1|^{2}\cdots |\varphi_n|^{2}\right)^J\mrd A_\beta(z).
\end{align*}
Therefore,
\begin{align*}
\sum_{J\ge 0}\|E^t_{\psi,\varphi}(e_J)\|^2_\beta &=\int_{\disk}|\psi(z)|^2\sum_{k=0}^\infty\frac{k^{2t}\,(k+n-1)!}{(n-1)!\, k!}|\varphi(z)|^{2k}\mrd A_\beta(z)\\
&\asymp\int_{\disk}|\psi(z)|^2\sum_{k=0}^\infty\frac{(k+n+2t-1)!}{(n+2t-1)!\, k!}|\varphi(z)|^{2k}\mrd A_\beta(z)\\
&=\int_{\disk}\frac{|\psi(z)|^2}{(1-|\varphi(z)|^2)^{n+2t}}\mrd A_\beta(z).
\end{align*}
\end{proof}

\section{Declarations}

\noindent{\bf Availability of data and materials}\\
Data sharing is not applicable to this article as no data sets were generated or analyzed.\\

\noindent{\bf Competing interests}\\
The author declares no competing interests.

\end{document}